\RequirePackage{amsmath}
\documentclass[runningheads]{llncs}
\usepackage[T1]{fontenc}
\usepackage{hyperref}       
\usepackage{url}            
\usepackage{booktabs}       
\usepackage{amsfonts}       
\usepackage{tablefootnote}
\usepackage{verbatim}
\usepackage{nicefrac}       
\usepackage{microtype}      
\usepackage{lipsum}
\usepackage{mathtools}

\usepackage{amsmath,amssymb}
\usepackage{algorithm,algorithmic}
\usepackage{pifont}
\usepackage{cases}
\usepackage{subcaption,graphicx}
\usepackage{stackengine}    
\usepackage{wrapfig}
\usepackage{enumitem}

\newtheorem{assumption}[theorem]{Assumption}

\newcommand{\eps}{\varepsilon}

\newcommand{\one}{\mathbf{1}}

\newcommand{\circledOne}{\text{\ding{172}}}
\newcommand{\circledTwo}{\text{\ding{173}}}

\newcommand{\numberthis}{\addtocounter{equation}{1}\tag{\theequation}}

\DeclareMathOperator*{\argmin}{arg\,min}

\DeclareMathOperator{\prox}{prox}

\DeclareMathOperator{\col}{col}
\DeclareMathOperator{\diag}{diag}

\newcommand{\R}{\mathbb{R}}

\newcommand{\mA}{{\bf A}}
\newcommand{\mB}{{\bf B}}

\newcommand{\mE}{{\bf E}}

\newcommand{\mI}{{\bf I}}

\newcommand{\mW}{{\bf W}}

\newcommand{\cA}{{\mathcal{A}}}

\newcommand{\cE}{{\mathcal{E}}}

\newcommand{\cG}{{\mathcal{G}}}

\newcommand{\cV}{{\mathcal{V}}}

\newcommand{\ba}{{\bf a}}
\newcommand{\bb}{{\bf b}}

\newcommand{\bp}{{\bf p}}
\newcommand{\bq}{{\bf q}}
\newcommand{\br}{{\bf r}}
\newcommand{\bs}{{\bf s}}
\newcommand{\bt}{{\bf t}}
\newcommand{\bu}{{\bf u}}

\newcommand{\bx}{{\bf x}}
\newcommand{\by}{{\bf y}}
\newcommand{\bz}{{\bf z}}

\newcommand{\ds}{\displaystyle}
\newcommand{\norm}[1]{\left\| #1 \right\|}

\newcommand{\angles}[1]{\left\langle #1 \right\rangle}
\newcommand{\cbraces}[1]{\left( #1 \right)}
\newcommand{\sbraces}[1]{\left[ #1 \right]}

\def\<#1,#2>{\langle #1,#2\rangle}

\newcommand{\sigmamax}{\sigma_{\max}(\cA)}
\newcommand{\sigmamaxsqr}{\sigma_{\max}^2(\cA)}
\newcommand{\sigmaminplus}{\sigma_{\min}^+(\cA)}

\usepackage[colorinlistoftodos,bordercolor=blue,backgroundcolor=blue!20,linecolor=blue,textsize=scriptsize]{todonotes}

\usepackage{color}

\begin{document}
	
\title{On Decentralized Nonsmooth Optimization\thanks{The research is supported by the Ministry of Science and Higher Education of the Russian Federation (Goszadaniye) 075-00337-20-03, project No. 0714-2020-0005.}}

\titlerunning{Decentralized Non-smooth Optimization}

\author{Savelii Chezhegov\inst{1}\orcidID{0009-0003-7378-3210} \and
  Alexander Rogozin\inst{1}\orcidID{0000-0003-3435-2680} \and
  Alexander Gasnikov\inst{1,2,3}\orcidID{0000-0002-7386-039X}}

\authorrunning{S. Chezhegov et al.}

\institute{
	Moscow Institute of Physics and Technology, Moscow, Russia \and
	Institute for Information Transportation Problems, Moscow, Russia \and
	Caucasus Mathematic Center of Adygh State University, Moscow, Russia
}

\maketitle              

\begin{abstract}
	In decentralized optimization, several nodes connected by a network collaboratively minimize some objective function. For minimization of Lipschitz functions lower bounds are known along with optimal algorithms. We study a specific class of problems: linear models with nonsmooth loss functions. Our algorithm combines regularization and dual reformulation to get an effective optimization method with complexity better than the lower bounds.

\end{abstract}


\section{Introduction}

The focus of this work is a particular class of problems in decentralized non-smooth optimization. We assume that each of computational agents, or nodes, holds a part of a common optimization problem and the agents are connected by a network. Each node may communicate with its immediate neighbors, and the agents aim to collaboratively solve an optimization problem.

On the class of smooth (strongly) convex functions endowed with a first-order oracle, decentralized optimization can be called a theoretically well-developed area of research. For this setting, \cite{scaman2017optimal} proposed lower bounds and optimal dual algorithms. After that, optimal gradient methods with primal oracle were developed in \cite{kovalev2020optimal}. Even if the network is allowed to change, lower bounds and optimal algorithms are known and established in a series of works \cite{kovalev2021adom,kovalev2021lower,li2021accelerated}.

However, the case when local functions are non-smooth is not that well studied. Algorithm proposed in \cite{scaman2018optimal} uses a gradient approximation via Gaussian smoothing. Such a technique results in additional factor of dimension. Distributed subgradient methods \cite{nedic2009distributed} are not optimal and only converge to a neighborhood of the solution if used with a constant step-size. In other words, development of an optimal decentralized algorithm for networks is an open research question.

As noted below, we restrict our attention to a particular class of decentralized non-smooth optimization problems. Namely, we study linear models with non-smooth loss functions and an entropic regularizer. Problems of such type arise in traffic demands matrix calculation \cite{anikin2015modern,zhang2005estimating}, optimal transport \cite{li2023fast} and distributed training with decomposition over features \cite{boyd2011distributed}.

\noindent\textbf{Traffic problems}. Following the arguments in \cite{anikin2015modern}, for example, one seeks to minimize $g(x)$ subject to constraints $Ax = b$. Here function $g(x)$ may be interpreted as some similarity measure between $x$ and a supposed solution. Moving the constraint $Ax = b$ as a penalty into the objective, we obtain a problem of type
\begin{align*}
	\min_{x\in\R^d}~ g(x) + \lambda \norm{Ax - b},
\end{align*}
where $\norm{Ax - b}$ denotes some norm. If the $g$ represents a similarity measure given by KL divergence, we obtain an optimization problem for linear model with entropic regularizer.

\noindent\textbf{Optimal transport}. Another example is entropy-regularized optimal transport \cite{li2023fast}. In paper \cite{li2023fast} the authors show that an optimal transportation problem can be rewritten as
\begin{align*}
	\min_{\bx\in\Delta_n^n}\min_{\by\in\Delta_2^n}~ \angles{\bx, \ba} - \angles{\by, \bb} + \angles{\mA\bx, \by} + \lambda_\bx \angles{\bx, \log\bx} - \lambda_\by\angles{\by, \log\by},
\end{align*}
where $\Delta_n$ denotes denotes a unit simplex of dimension $n$. This illustrated that entropy-regularized linear models can arise in saddle-point optimization, as well.

\noindent\textbf{Distributed ML}. In distributed statistical inference and machine learning one may want to train a model in a distributed way \cite{boyd2011distributed}. Consider a dataset with a moderate number of training examples and a large number of features. Let the dataset be split between the nodes not by samples but by features. Let $\ell$ be the common loss function, and for each agent $i$ introduce its local dataset $(A_i, b_i)$ and the corresponding regularizer $r_i(x)$. That leads to a fitting problem
\begin{align*}
	\min_{x_1, \ldots, x_m}~ \ell\cbraces{\sum_{i=1}^m A_i x_i - b_i} + \sum_{i=1}^m r_i(x_i).
\end{align*}


\noindent\textbf{Our contribution}. In our work we propose a dual algorithm for non-smooth decentralized optimization. The dual problem is smooth although the initial one is non-smooth, but also subject to constraints. The constraints can be equivalently rewritten as a regularizer. We show that a resulting regularized problem can be solved by an accelerated proximal primal-dual gradient method. 

We study a specific class of problems, and our approach allows to break the lower bounds in \cite{scaman2018optimal}. Omitting problem parameters, the iteration and communication complexities of our algorithm are $O(\sqrt{1/\eps})$, while lower bounds suggest that at least $\Omega(1/\eps)$ communication rounds and at least $\Omega(1/\eps^2)$ local computations at each node are required.





\subsection{Notation}

Let $\otimes$ denote the Kronecker product. Let $\Delta_d = \{x\in\R^d:~ \sum_{i=1}^m x_i = 1,~ x_i\geq 0,~ i = 1, \ldots, m\}$ be a unit simplex in $\R^d$. By $\Delta_d^m$ we understand a product of $m$ simplices that is a set in $\R^{md}$. For $p \geq 0$, let $\norm{x}_p = \cbraces{\sum_{i=1}^d |x_i|^p}^{1/p}$ denote the $p$-norm in $\R^d$. By $\angles{a, b}$ we denote a scalar product of vectors. Also let $\bx = \col[x_1, \ldots, x_m] = (x_1^\top \ldots x_m^\top)^\top\in\R^{md}$ denote a column vector stacked of $x_1, \ldots, x_m\in\R^d$. Similarly for matrices $C_1, \ldots, C_m\in\R^{n\times d}$, introduce $\col[C_1, \ldots, C_m] = (C_1^\top\ldots C_m^\top)^\top\in\R^{mn\times d}$. Moreover, let $\diag[C_1, \ldots, C_m]$ denote a block matrix with blocks $C_1, \ldots, C_m$ at the diagonal. For $x\in\R^d$, let $\log x$ denote a natural logarithm function applied component-wise. We define $\one_n$ to be a vector of all ones of length $n$ and $\mI_n$ to be an identity matrix of size $n\times n$. Also denote the $i$-th coordinate vector of $\R^n$ as $e_i^{(n)}$. Let $\lambda_{\max}(C)$ and $\lambda_{\min}^+(C)$ denote the maximal and minimal nonzero eigenvalue of matrix $C$. Let $\sigma_{\max}(C)$ and $\sigma_{\min}^+(C)$ denote the maximal and minimal nonzero singular values of $C$. 

Given a convex closed set $Q$, let $\Pi_Q$ denote a projection operator on it and denote its interior $\text{int}~Q$. For a closed proper function $h(x):~Q\to\R$ and a scalar $\gamma > 0$, define proximal operator as
\begin{align*}
	\prox_{\gamma h}(x) = \argmin_{y\in Q}\cbraces{h(x) + \frac{1}{2\gamma}\norm{y - x}_2^2}.
\end{align*}

\section{Problem and assumptions}
Consider $m$ independent computational entities, or agents. Agent $i$ locally holds a dataset consisting of matrix $A_i$ and labels $b_i$. Let $A = \col[A_1, \ldots, A_m]\in\R^{n\times d}$ be the training samples and $\bb = \col[b_1, \ldots, b_m]$ be the labels. The whole training dataset $(A, \bb)$ is distributed between $m$ different machines. We consider $p$-norm minimization over unit simplex with entropy regularizer.
\begin{align}\label{eq:problem_initial}
	\min_{x\in\Delta_d} \frac{1}{m}\norm{Ax - \bb}_p + \theta\angles{x, \log x},
\end{align}
where $\theta > 0$ is a regularization parameter.

The agents can communicate information through a communication network. We assume that each machine is a node in the network that is represented by a connected undirected graph $\cG = (\cV, \cE)$. The nodes can communicate if and only if they are connected by an edge.

With graph $\cG$ we associate a communication matrix $W$ that has the following properties.
\begin{assumption}
\item
1. (Network compatibility) $[W]_{ij} = 0$ if $i\neq j$ and $(i, j)\notin\cE$.\\
2. (Positive semi-definiteness and symmetry) $W\succeq 0,~ W^\top = W$.\\  
3. (Kernel property) $Wx = 0$ if and only if $x_1 = \ldots = x_m$.\\
\end{assumption}
We also introduce the condition number of the communication matrix.
\begin{align}\label{eq:def_chi}
	\chi = \frac{\lambda_{\max}(W)}{\lambda_{\min}^+(W)}.
\end{align}

In order to get a distributed formulation, assign each agent $i$ in the network a local copy of the solution vector $x_i$. Define $\bx = \col[x_1, \ldots, x_m]$, $\mA = \diag[A_1, \ldots, A_m]$ and introduce $\by = \mA\bx$.
\begin{align}\label{eq:primal_problem_new_variables}
	\min_{\bx\in\Delta_m^d} &\norm{\by - \bb}_p + \theta\angles{\bx, \log\bx} \\
	\text{s.t. } &\mW\bx = 0 \nonumber \\
	&\by = \mA\bx \nonumber
\end{align}
The complexity of distributed methods typically depends on the condition number of the communication matrix (it is $\chi$ defined in \eqref{eq:def_chi}) and on condition numbers of objective functions. For brevity we introduce
\begin{align}\label{eq:def_sigmamax_sigmaminplus}
	\sigmamax = \max_{i=1,\ldots,m}\cbraces{\sigma_{\max}(A_i)},~ \sigmaminplus = \min_{i=1,\ldots,m} \sigma_{\min}^+(A_i).
\end{align}


\section{Dual problem}

Let us derive a dual problem to \eqref{eq:primal_problem_new_variables}. It is convenient to introduce $F(\by) = \norm{\by - \bb}_p,~ G(\bx) = \theta\angles{\bx, \log\bx}$.

\subsection{Conjugate functions}\label{2}
Let us derive the conjugate functions $F^*$ and $G^*$. Let $q \geq 1$ be such that $\frac{1}{p} + \frac{1}{q} = 1$.
\begin{align*}
    F^*(\bt) = \sup_{\by\in\R^{mn}} (\angles{\bt, \by} - F(\by)) 
    &= \sup_{\by\in\R^{mn}} (\angles{\bt, \by - \bb} - \|\by - \bb\|_p) + \angles{\bt, \bb} \\
    &= \sup_{\br\in\R^{mn}} (\angles{\bt, \br} - \|\br\|_p) + \angles{\bt, \bb} \\
    &= \begin{cases}
        \angles{\bt, \bb}, & \|\bt\|_{q} \leq 1 \\
        +\infty, & \mbox{otherwise}
        \end{cases}
\end{align*}
Last equation is a result of conjugate function for $\|x\|_p$, which is taken from a classical book by Boyd \cite{boyd2004convex}, Chapter 5.

In order to compute $G^*$, introduce $g(x) = \theta\angles{x, \log x}:~ \R^d\to\R$ and note that $G(\bx) = \sum_{i=1}^m g_i(x_i)$.
\begin{align*}
    g^*(t) = \sup_{x\in \Delta_d} (\<t, x> - \theta \<x, \log(x)>) 
\end{align*}
Writing a Lagrange function:
\begin{align*}
    L(t, x) &= \angles{t, x} - \theta\angles{x, \log(x)} + \lambda \cbraces{\one_d^\top x - 1}\\
    \nabla_x L(t, x) &= t - \theta\log x - \theta\one_d + \lambda\one_d = 0 \Rightarrow x = \exp\cbraces{\frac{t}{\theta} + \one_d\cbraces{\frac{\lambda}{\theta} - 1}}\\
    \one_d^\top x &= 1 \Rightarrow \exp\cbraces{\frac{\lambda}{\theta} - 1} \one_d^\top \exp\cbraces{\frac{t}{\theta}} = 1 \Rightarrow \exp\cbraces{\frac{\lambda}{\theta} - 1} = \frac{1}{\one_d^\top \exp\cbraces{\frac{t}{\theta}}}
\end{align*}
As a consequence
\begin{align*}
    &x = \frac{\exp\cbraces{\frac{t}{\theta}}}{\one_d^\top \exp\cbraces{\frac{t}{\theta}}}
\end{align*}
\noindent Using equation to $x$,
$$
g^*(t) = \theta \log\cbraces{\one_d^\top\exp\cbraces{\frac{t}{\theta}}}
$$
As noted above, $G(\bx)$ is separable, i.e. $G(\bx) = \sum_{i=1}^m g(x_i)$. Therefore,
\begin{align*}
	G^*(\bt)
	= \sup_{\bx\in\Delta_d^m}\cbraces{\angles{\bt, \bx} - \sum_{i=1}^m g(x_i)}
	= \sum_{i=1}^m \sup_{x\in\Delta_d} \cbraces{\angles{t_i, x} - g(x)}
	= \sum_{i=1}^m g^*(t_i).
\end{align*}
It is convenient to express $t_i$ through $\bt$. Introduce matrix 
\begin{align}
	\mE_i = \cbraces{e_i^{(m)}}^\top\otimes\mI = [0\ldots 0~\mI~0\ldots 0].
\end{align} Then $t_i = \mE_i\bt$. It holds
\begin{align*}
	G^*(\bt) = \sum_{i=1}^m g^*(\mE_i\bt).
\end{align*}

\subsection{Dual problem formulation}\label{1}

Let us derive a dual problem to \eqref{eq:primal_problem_new_variables}. It is convenient to denote $F(\by) = \norm{\by - \bb}_p,~ G(\bx) = \theta\angles{\bx, \log\bx}$. Introduce dual function
\begin{align*}
\Phi(\bz, \bs)
&= \inf_{\bx\in \Delta_d^m, \by \in \mathbb{R}^n} \sbraces{F(\by) + G(\bx) + \angles{\bz, \mW\bx} + \angles{\bs, \mA\bx-\by}}\\
&=  \inf_{\by\in\R^{mn}}\sbraces{F(\by) - \angles{\bs, \by}} + \inf_{\bx\in \Delta_d^m} \sbraces{G(\bx) + \angles{\mW\bz + \mA^\top\bs, \bx}} \\
&=  -\sup_{\by\in\R^{mn}} \sbraces{\angles{\bs, \by} - F(\by)} - \sup_{\bx\in\Delta_d^m} \sbraces{\angles{-\mW\bz - \mA^\top\bs, \bx} - G(\bx)} \\
&= -F^*(\bs) - G^*(-\mW\bz -\mA^\top\bs)
\end{align*}
As a consequence, dual problem can be formulated as
\begin{align*}
\min_{\bz\in\R^{md}, \bs\in\R^{mn}}~ F^*(\bs) + G^*(-\mW\bz -\mA^\top\bs).
\end{align*}
Results from \ref{1} and \ref{2} leads us to final dual problem formulation
\begin{align}
\label{formula: original}
    \min_{\bz, \bs:\|\bs\|_{q}\leq 1} \<\bs, \bb> + \sum\limits_{i=1}^m \theta \log\cbraces{\one_d^\top\exp\cbraces{-\frac{1}{\theta}\mE_i\cbraces{\mW\bz + \mA^\top\bs}}}
\end{align}
The constrained problem above is equivalent to a regularized problem
\begin{align}
\label{formula: regularization}
    \min_{\bz, \bs}~ \angles{\bs, \bb} + \sum\limits_{i=1}^m \theta \log\cbraces{\one_d^\top\exp\cbraces{-\frac{1}{\theta}\mE_i(\mW\bz + \mA^\top\bs)}} + \nu\|\bs\|^q_q,
\end{align}
where $\nu > 0$ is a scalar.\\
As a result, the dual problem writes as
\begin{align*}
	\min_{\bq}~ &H(\bz, \bs) + R(\bz, \bs) \numberthis\label{eq:dual_problem_with_regularizer} \\
	H(\bz, \bs) &= \angles{\bs, \bb} + \sum\limits_{i=1}^m \theta \log\cbraces{\one_d^\top\exp\cbraces{-\frac{1}{\theta}\mE_i(\mW\bz + \mA^\top\bs)}} \\
	R(\bz, \bs) &= \nu\norm{\bs}_q^q.
\end{align*}
Recall problem \eqref{eq:dual_problem_with_regularizer} and denote $\mB = (-\mW~ -\mA^\top),~ \bq = \col[\bz, \bs],~ \bp = \col[0, \bb]$. With slight abuse of notation we write $H(\bq) = H(\bz, \bs)$ and $R(\bq) = R(\bz, \bs)$.
Problem \eqref{eq:dual_problem_with_regularizer} takes the form
\begin{align*}
	\min_{\bq}~ H(\bq) + R(\bq).
\end{align*}
Here $H$ is a differentiable function and $R$ is a regularizer, or composite term. Problems of such type are typically solved by proximal optimization methods.

\section{Algorithms and Complexities}

\subsection{Similar Triangles Method}

We apply an accelerated primal-dual algorithm called Similar Triangles Method (STM) \cite{dvurechensky2017adaptive}.

\begin{algorithm}[H]
	\caption{Similar triangles method(STM)}
	\label{alg:STM}
	\begin{algorithmic}[1]
		\REQUIRE{$A_0 = \alpha_0 = 0,~ \bq^0 = \bu^0 = \by^0$.}
		\FOR{$k = 0, \ldots, N - 1$}
		\STATE{Find $\alpha_{k+1}$ from equality $(A_{k} + \alpha_k)(1 + A_k\mu) = L\alpha_{k+1}$ and put $A_{k+1} = A_k + \alpha_{k+1}$.}
        \STATE{Introduce $$\phi_{k+1}(\bx) = \alpha_{k+1}\cbraces{\angles{\nabla H(\by^{k+1}), \bx} + R(\bx)} + \frac{1 + A_k\mu}{2}\norm{\bx - \bu^k}_2^2 + \frac{\mu\alpha_{k+1}}{2}\norm{\bx - \by^{k+1}}_2^2$$}
		\STATE{$\by^{k+1} = \frac{\alpha_{k+1}\bu^k + A_k\bq^k}{A_{k+1}}$}
		\STATE{$\ds\bu^{k+1} = \argmin_{\bq} \sbraces{\phi_{k+1}(\bx)}$\label{alg:STM:update}}
		\STATE{$\bq^{k+1} = \frac{\alpha^{k+1}\bu^{k+1} + A_k\bq^k}{A_{k+1}}$}
		\ENDFOR
	\end{algorithmic}
\end{algorithm}
First, note that line~\ref{alg:STM:update} of Algorithm~\ref{alg:STM} can be decomposed into a gradient step and computation of proximal operator of $R$.
\begin{align*}
	 \bu^{k+1} = \argmin_{\bq}\sbraces{\phi_{k+1}(\bx)} = \prox_{\gamma_k R} \sbraces{\mu\gamma_k\by^{k+1} + (1 - \mu\gamma_k)\bu^k - \gamma_k\nabla F(\by^{k+1})},
\end{align*}
where $\gamma_k = \frac{\alpha_{k+1}}{1 + \mu A_{k+1}}$. Let us show that this operator can be easily computed. Let $\bt = \col[\bt_\bz, \bt_\bs]$. By definition of proximal operator we have
\begin{align*}
	\prox_{\gamma_k R}(\bt) 
	&= \argmin_{\bs}\cbraces{\frac{1}{2\gamma_k}\norm{\bt - \bs}_2^2 + R(\bq)} \\
	&= \argmin_{\bz, \bs}\cbraces{\frac{1}{2\gamma_k}(\norm{\bt_\bs - \bs}_2^2 + \norm{\bt_\bz - \bz}_2^2) + \nu\norm{\bs}_q^q}.
\end{align*}
Let $\tilde\bq = \col[\tilde\bz, \tilde\bs] = \prox_{\gamma_k R}(\bt)$. We have $\tilde\bz = \bt_\bz$. Let $\tilde s_i$ denote the $i$-th component of $\tilde\bs$ and $t_i$ denote the $i$-th component of $\bt_\bs$; then $\tilde s_i$ can be found from equation
\begin{align*}
	t_i - \tilde s_i + \gamma_k q\nu |\tilde s_i|^{q - 1} = 0.
\end{align*}
We assume that the equation above can be efficiently numerically solved w.r.t. $\tilde s_i$. For example, it can be done by solution localization methods such as binary search. As a result, we see that the proximal operator of $R$ can be computed cheaply.


\noindent Let us formulate the theorem on convergence of Algorithm~\ref{alg:STM} for the problem \eqref{eq:dual_problem_with_regularizer}.
\begin{theorem}
\label{thm: STM-complexity}
    Algorithm \ref{alg:STM} requires 
    $$
    O\cbraces{\cbraces{\frac{\theta m\cbraces{\sigmamaxsqr + \sigma_{\max}^2(W)}\|\log x^* + \one_d\|^2_2}{\min((\sigma^+_{min}(\cA))^2, (\lambda^+_{min}(W))^2)\eps}}^{1/2}}
    $$
    iterations to reach $\eps$-accuracy for the problem \eqref{eq:dual_problem_with_regularizer}
\end{theorem}
Before we prove the above result, we need to formulate some lemmas.\\
We need to find Lipschitz constant for dual problem. Namely, let us find the Lipschitz constant for function $G^*(-\mW\bz - \mA^\top\bs)$ as a function of $\bq = \col[\bz, \bs]$.
\begin{lemma}
\label{lem: smoothness}
	Function $H(\bq)$ has a Lipschitz gradient with constant
	\begin{align*}
		L_H = \frac{m\cbraces{\sigmamaxsqr + \sigma_{\max}^2(W)}}{\theta}
	\end{align*}
\end{lemma}
\begin{proof}
According to \cite{kakade2009duality}, if a function is $\mu$-strongly convex in norm $\norm{\cdot}_2$, then its conjugate function $h^*(y)$ has a $\frac{1}{\mu}$-Lipschitz gradient in $\norm{\cdot}_2$.\\
Using the fact from \cite{boyd2004convex}, Chapter 3, we obtain that the conjugate function of
$$h(x) = \log\left(\sum\limits_{i=1}^d \exp(x_i)\right)$$
is $$h^*(y) = 
\begin{cases}
    \angles{y, \log{y}}, & y \in \Delta_d\\
    \infty, & \text{otherwise}
\end{cases}$$
To have a constant of strongly convexity, we can find a minimal eigenvalue of Hessian of $h^*(y)$
$$\nabla^2 h^*(y) = \text{diag}\cbraces{\frac{1}{y_1}, \ldots, \frac{1}{y_d}}.$$
For any $y\in\text{int}\Delta_d$, we have that $1/y_i\geq 1,~ i = 1, \ldots, d$. Therefore, we have $\lambda_{\min}(\nabla^2 h^*(y))\geq 1$, i.e. $\mu_{h^*} \geq 1$.\\
As a consequence, for function $$h(x) = \log\left(\sum\limits_{i=1}^d \exp(x_i)\right)$$
Lipschitz constant is equal to $L_h = 1$.\\
Therefore, for a function $$g^*(x) = \theta h\Big(\frac{x}{\theta}\Big)$$
Lipschitz constant is equal to $L_g = 1/\theta$.\\
Introduce $\mB = \Big(-\mW, -\mA^\top\Big)$. We have
\begin{align*}
H(\bq) = G^*(\mB\bq) + \angles{\bs, \bb} = \sum_{i=1}^m g^*(\mE_i\mB\bq) + \angles{\bs, \bb}.
\end{align*}
It holds
\begin{align*}
\|\nabla &H(\bq_2) - \nabla H(\bq_2)\|_2 = \|\mB^\top\nabla G^*(\mB\bq_2) - \mB^\top\nabla G^*(\mB\bq_2)\|_2 \\
&\leq \sigma_{\max}(\mB)\norm{\nabla G^*(\mB\bq_2) - \nabla G^*(\mB\bq_1)}
\leq \sigma_{\max}(\mB)\sum_{i=1}^m \norm{\nabla g^*(\mE_i\mB\bq_2) - \nabla g^*(\mE_i\mB\bq_1)}_2 \\
&\leq \sigma_{\max}(\mB)\sum_{i=1}^m \frac{\sigma_{\max}(\mE_i\mB)}{\theta}\norm{\bq_2 - \bq_1}_2
\leq \frac{m\sigma_{\max}^2(\mB)}{\theta} \norm{\bq_2 - \bq_1}_2 \\
&= \frac{m(\sigma_{\max}^2(\mA) + \sigma_{\max}^2(\mW))}{\theta}\norm{\bq_2 - \bq_1}_2 \\
&\leq \frac{m\cbraces{\underset{i=1,\ldots,m}{\max}(\sigma_{\max}^2(A_1), \ldots, \sigma_{\max}^2(A_m)) + \sigma_{\max}^2(W)}}{\theta}\norm{\bq_2 - \bq_1}_2 \\
&= L_H\norm{\bq_2 - \bq_1}_2,
\end{align*}
which finishes the proof of lemma.
\end{proof}
For writing a complexity of solver for our problem, we also need to bound the dual distance.
\begin{lemma}
    \label{lem:solution norm}
	Let $\bq^* = \col[\bz^*, \bs^*]$ be the solution of dual problem~\eqref{eq:dual_problem_with_regularizer} and let $x^*$ be a solution of~\eqref{eq:problem_initial}. It holds 
	\begin{align*}
		\norm{\bq^*}_2^2\leq R_{dual}^2 = \frac{\theta^2 m\|\log x^* + \one_d\|^2_2}{\min((\sigma^+_{min}(\cA))^2, (\lambda^+_{min}(W))^2)}.
	\end{align*}
\end{lemma}
\begin{proof} Let $(\bx^*, \by^*)$ be a solution to primal problem~\eqref{eq:primal_problem_new_variables}. In particular, we have $\bx^* = \one_m\otimes x^*$. Then $(\bx^*, \by^*, \bz^*, \bs^*)$ is a saddle point of Lagrange function. For any $\bx\in\Delta_m^d,~ \by\in\R^{md},~ \bz\in\R^{md}, \bs\in\R^{mn}$ it holds
\begin{align*}
    F(\by^*) + &G(\bx^*) + \<\bz, \mW\bx^*> + \<\bs, \mA\bx^* - \by^*> \\
    &\leq F(\by^*) + G(\bx^*) + \<\bz^*, \mW\bx^*> + \<\bs^*, \mA\bx^* - \by^*> \\
    &\leq F(\by) + G(\bx) + \<\bz^*, \mW\bx> + \<\bs^*, \mA\bx - \by>
\end{align*}
Substituting $\by = \by^*$ we obtain
\begin{align*}
    &G(\bx) \geq G(\bx^*) + \<-\mW\bz^* - \mA^\top\bs^*, \bx - \bx^*>\\
    &-\mW\bz^* - \mA^\top\bs^* = \nabla G(\bx^*)
\end{align*}
Recalling that $\mB = (-\mW, -\mA^\top)$ we derive
\begin{align*}
    \mB\bq^* &= \nabla G(\bx^*)\\
    \<\mB^\top\mB\bq^*, \bq^*> &= \|\nabla G(\bx^*)\|^2_2\\
    \lambda^+_{min}(\mB^\top\mB)\|\bq^*\|^2_2 &\leq \|\nabla G(\bx^*)\|^2_2\\
    \|\bq^*\|^2_2 &\leq \frac{\|\nabla G(\bx^*)\|^2_2}{\lambda^+_{min}(\mB^\top\mB)} 
\end{align*}
We have
\begin{align*}
\lambda_{\min}^+(\mB^\top \mB)
&= \lambda_{\min}^+(\mB\mB^\top)
= \lambda_{\min}^+ (\mW^2 + \mA^\top \mA)
= \lambda_{\min}^+(\mW^2\otimes\mI_d + \mI_m\otimes \mA^\top \mA) \\
&= \min((\lambda_{\min}^+(\mW))^2, (\sigma_{\min}^+(\mA))^2)
= \min\cbraces{(\lambda_{\min}^+(W))^2, (\sigma_{\min}^+(\cA))^2}
\end{align*}
and
\begin{align*}
	\norm{\nabla G(\bx^*)}_2^2 = \theta^2 \norm{\log\bx^* + \one_{md}}_2^2 = \theta^2m\norm{\log x^* + \one_d}_2^2.
\end{align*}
As a result, we obtain
\begin{align*}
	R^2_{dual} = \frac{\theta^2 m\|\log x^* + \one_d\|^2_2}{\min((\sigma^+_{min}(\cA))^2, (\lambda^+_{min}(W))^2)}.
\end{align*}
\end{proof}
Now we prove the theorem about complexity of Similar Triangles Method.\\
\\
\textbf{Proof of Theorem \ref{thm: STM-complexity}}
\begin{proof} 
    First, note that solution accuracy $\eps$ for problem~\eqref{eq:problem_initial} is equivalent to accuracy $m\eps$ for problem~\eqref{eq:dual_problem_with_regularizer}. STM requires $O((L_H R_{dual}^2/(m\eps))^{1/2})$ iterations to reach $\eps$-accuracy. Combining the results from lemmas \ref{lem: smoothness} and \ref{lem:solution norm} we obtain the final complexity.
\end{proof}

\subsection{Accelerated block-coordinate method}
In previous section, our approach was based on a way where we apply a first-order method without separating the variables. But we can treat variable blocks $\bz \text{ and } \bs$ separately and get a better convergence bound. We apply an accelerated method ACRCD (Accelerated by Coupling Randomized Coordinate Descent) from \cite{gasnikov2016nontriviality}. We describe the result only for the case $p = 1$. In this case, we apply ACRCD not to regularized dual problem~\eqref{eq:dual_problem_with_regularizer}, but to constrained version of dual problem \eqref{formula: original}. We also note that ACRCD is primal-dual, so solving the dual problem with accuracy $\eps$ is sufficient to restore the solution of the primal with accuracy $\eps$.
\begin{algorithm}[H]
\caption{ACRCD}
\label{alg:ACRCD}
\begin{algorithmic}[1]
	\REQUIRE{Define coefficients $\alpha_{k+1} = \frac{k+2}{8},~ \tau_k = \frac{2}{k+2}$. Choose stepsizes $L_\bz,~ L_\bs$. Put $\overline{\bz}^0 = \underline{\bz}^0 = \bz^0,~ \overline{\bs}^0 = \underline{\bs}^0 = \bs^0$.}
	\FOR{$k = 0, 1, \ldots, N - 1$}
	\STATE{$\bz^{k+1} = \tau_k \underline{\bz}^k + (1 - \tau_k)\overline\bz^k$}
	\STATE{$\bs^{k+1} = \tau_k \underline{\bs}^k + (1 - \tau_k)\overline\bs^k$}
	\STATE{Put $\xi_i = 1$ with probability $\eta$ and $\xi = 0$ with probability $(1-\eta)$, where $\eta = \frac{\lambda_{\max}(W)}{\lambda_{\max}(W) + \sigmamax}$}
	\vspace{0.2cm}
	\IF{$\xi_i = 1$}
	\STATE{$\overline{\bz}^{k+1} = \bz^{k+1} - \frac{1}{L_\bz}\nabla H_\bz(\bz^{k+1}, \bs^{k+1})$}
	\STATE{$\underline{\bz}^{k+1} = \underline{\bz}^k - \frac{2\alpha_{k+1}}{L_\bz}\nabla H_\bz(\bz^{k+1}, \bs^{k+1})$}
	\ELSE
	\STATE{$\overline{\bs}^{k+1} = \Pi_{[-1, 1]^{mn}}\sbraces{\bs^{k+1} - \frac{1}{L_\bs}\nabla H_\bs(\bz^{k+1}, \bs^{k+1})}$}
	\STATE{$\underline{\bs}^{k+1} = \Pi_{[-1, 1]^{mn}}\sbraces{\underline{\bs}^k - \frac{2\alpha_{k+1}}{L_\bs}\nabla H_\bs(\bz^{k+1}, \bs^{k+1})}$}
	\ENDIF
	\ENDFOR
\end{algorithmic}
\end{algorithm}
\begin{theorem}\label{th:acrcd_convergence}
	To reach accuracy $\eps$ with probability at least $(1 - \delta)$, Algorithm~\ref{alg:ACRCD} requires $N_{comm}$ communication rounds and $N_{comp}$ local computations, where
	\begin{align*}
		N_{comm} &= \frac{m^{1/4}}{\sqrt{\theta\eps}} \frac{\lambda_{\max}(W)}{\lambda_{\min}^+(W)}\cbraces{2\theta^2 \norm{\log\bx^* + \one_d}_2^2 + 2n\sigmamaxsqr + n(\lambda_{\min}^+(W))^2}^{1/2}\log\cbraces{\frac{1}{\delta}}, \\
		N_{comp} &= \frac{m^{1/4}}{\sqrt{\theta\eps}} \frac{\sigmamax}{\lambda_{\min}^+(W)}\cbraces{2\theta^2 \norm{\log\bx^* + \one_d}_2^2 + 2n\sigmamaxsqr + n(\lambda_{\min}^+(W))^2}^{1/2}\log\cbraces{\frac{1}{\delta}}.
	\end{align*}
\end{theorem}


First, we need to estimate Lipschitz constants for gradients of each block of variables. If we consider the function $H$ as a function of two blocks of variables, the next result follows.
\begin{lemma}
\label{lem:smoothness_of_blocks}
    Function $H(\bz, \bs)$ has a $L_\bz$-Lipschitz gradient w.r.t. $\bz$ and $L_\bs$-Lipschitz gradient w.r.t. $\bs$, where
	\begin{align*}
		L_\bz &= \frac{\sqrt m\sigma_{\max}^2(W)}{\theta},~ L_\bs = \frac{\sqrt m\sigmamaxsqr}{\theta}.
	\end{align*}
\end{lemma}
\begin{proof}
	Recall that we denoted $\bx = \col[x_1, \ldots, x_m]$ and consider $\bs_1, \bs_2\in\R^{nm}$. Also denote $[\bx]_i = E_i\bx = x_i$.
	\begin{align*}
		\|\nabla_\bs &H(\bz, \bs_2) - \nabla_\bs H(\bz, \bs_1)\|_2 \\
		&= \|\mA\nabla G^*(-\mW\bz - \mA^\top\bs_2) - \mA\nabla G^*(-\mW\bz - \mA^\top\bs_1)\|_2 \\
		&\leq \sum_{i=1}^m \norm{A_i\nabla g^*\cbraces{-[\mW\bz]_i - [\mA^\top\bs_2]_i} - A_i\nabla g^*\cbraces{-[\mW\bz]_i - [\mA^\top\bs_1]_i}}_2 \\
		&\overset{\circledOne}{=} \sum_{i=1}^m \norm{A_i\nabla g^*\cbraces{-[\mW\bz]_i - A_i^\top[\bs_2]_i} - A_i\nabla g^*\cbraces{-[\mW\bz]_i - A_i^\top[\bs_1]_i}}_2 \\
		&\leq \frac{\sigmamax}{\theta}\sum_{i=1}^m \norm{A_i^\top[\bs_2]_i - A_i^\top[\bs_1]_i}_2
		\leq \frac{\sigmamaxsqr}{\theta} \sum_{i=1}^m \norm{[\bs_2]_i - [\bs_1]_i}_2 \\
		&\overset{\circledTwo}{\leq} \frac{\sqrt{m}\sigmamaxsqr}{\theta}\norm{\bs_2 - \bs_1}_2,
	\end{align*}
	where $\circledOne$ holds due to the structure of $\mA = \diag[A_1, \ldots, A_m]$ and $\circledTwo$ holds by convexity of the $2$-norm.

	Now consider the gradient w.r.t. $\bz$. Let $[\bx]^{(i)} = [x_1^{(i)}\ldots x_m^{(i)}]^\top$ denote a vector consisting of $i$-th components of $x_1, \ldots, x_m$. We have $[\mW\bx]_i = W[\bx]^{(i)}$ due to the structure of $\mW = W\otimes \mI_d$.
	\begin{align*}
		\|\nabla &H_\bz^*(\bz_2, \bs) - \nabla H_\bz^*(\bz_1, \bs)\|_2 \\
		&= \|\mW\nabla G^*(-\mW\bz_2 - \mA^\top\bs) - \mW\nabla G^*(-\mW\bz_1 - \mA^\top\bs)\|_2 \\
		&\leq \sum_{i=1}^m \norm{W\nabla g^*\cbraces{-W[\bz_2]^{(i)} - [\mA^\top\bs]_i} - W\nabla g^*\cbraces{-W[\bz_1]^{(i)} - [\mA^\top\bs]_i}}_2 \\
		&\leq \frac{\lambda_{\max}(W)}{\theta}\sum_{i=1}^m \norm{W([\bz_2]^{(i)} - [\bz_1]^{(i)})}_2
		\leq \frac{\lambda_{\max}^2(W)}{\theta} \sum_{i=1}^m \norm{[\bz_2]^{(i)} - [\bz_1]^{(i)}}_2 \\
		&\overset{\circledOne}{\leq} \frac{\sqrt{m}\lambda_{\max}^2(W)}{\theta}\norm{\bz_2 - \bz_1}_2,
	\end{align*}
	where $\circledOne$ holds by convexity of the $2$-norm.
\end{proof}
We need to bound dual distance for each block of variables, but first we need to claim an useful proposition from functional analysis.
\begin{proposition}
    \label{prop: norm-equivalence}
    Let $p > r \geq 1, x \in \mathbb{R}^d$. It holds
    \begin{align*}
        \|x\|_p \leq \|x\|_r \leq d^{\cbraces{\frac{1}{r}- \frac{1}{p}}}\|x\|_p
    \end{align*}
\end{proposition}
\begin{proof}
    This is a fairly well-known fact with a simple idea of proof. In fact, it is a direct consequence of H\'older's inequality, what means that constant in an inequality are unimprovable.
\end{proof}
Now we derive the bound on the norm of the dual solution. The convergence result only relies on the case $p = 1~ (q = \infty)$, but we derive a bound for any $q\geq 1$.
\begin{lemma}  
\label{lem: norm-block-solutions}
	Let $\bz^*, \bs^*$ be the solutions of dual problem~\eqref{eq:dual_problem_with_regularizer} and let $x^*$ be a solution of~\eqref{eq:problem_initial}. It holds 
	\begin{align*}
		\norm{\bz^*}_2^2 &\leq R_\bz^2 = \frac{2\theta^2m\|\log x^* + \one_d\|_2^2 + 2\sigmamaxsqr \cdot \max{\cbraces{1, \cbraces{mn}^{\cbraces{1 - \frac{2}{q}}}}}} {(\lambda^+_{min}(W))^2}\\	
        \norm{\bs^*}_2^2 &\leq R_\bs^2 = \max{\cbraces{1, \cbraces{mn}^{\cbraces{1 - \frac{2}{q}}}}}
	\end{align*}
\end{lemma}
\begin{proof}
    Using that the problems \ref{formula: original} and \ref{formula: regularization} are equal, that means
    \begin{align}
    \label{ineq: q-norm}
        \|\bs^*\|_q \leq 1
    \end{align}
    Using Proposition \ref{prop: norm-equivalence} we have 
    \begin{align}
    \label{ineq: dependence-on-q}
        \|\bs^*\|_2^2 \leq
        \begin{cases}
        \|\bs^*\|_q^2, & q < 2\\
        \cbraces{mn}^{\cbraces{1 - \frac{2}{q}}}\|\bs^*\|_q^2, & q \geq 2
        \end{cases}
    \end{align}
    Combininq \ref{ineq: q-norm} and \ref{ineq: dependence-on-q}, we have 
    \begin{align}
    \label{ineq: final-bound-s}
        \|\bs^*\|_2^2 \leq
        \begin{cases}
        1, & q < 2\\
        \cbraces{mn}^{\cbraces{1 - \frac{2}{q}}}, & q \geq 2
        \end{cases}
    \end{align}
    With the fact that $\cbraces{mn}^{\cbraces{1 - \frac{2}{q}}} < 1$ where $q < 2$ we state the claimed result.\\
    Using the fact from proof of Lemma \ref{lem:solution norm} such that
    \begin{align*}
       &-\mW\bz^* - \mA^\top\bs^* = \nabla G(\bx^*)
    \end{align*}
    we have
    \begin{align*}
        \|\mW\bz^*\|_2^2 \leq 2\|\nabla G(\bx^*)\|_2^2 + 2\|\mA^\top\bs^*\|_2^2 \leq 2\theta^2m\|\log x^* + \one_d\|_2^2 + 2\sigmamaxsqr \cdot \norm{\bs^*}^2_2
    \end{align*}
    As a result
    \begin{align}
    \label{ineq: bound-with-s}
        \|\bz^*\|_2^2 &\leq \frac{2\theta^2m\|\log x^* + \one_d\|_2^2 + 2\sigmamaxsqr \cdot \norm{\bs^*}^2_2}{(\lambda^+_{min}(W))^2}
    \end{align} 
    Using \ref{ineq: final-bound-s} into \ref{ineq: bound-with-s}, we claim the final result.
\end{proof}

\begin{proof}[Proof of Theorem~\ref{th:acrcd_convergence}]
The proof is based on results in \cite{gasnikov2016nontriviality}. We have two blocks of variables: $\bz$ and $\bs$. Firstly, Remark 3 of \cite{gasnikov2016nontriviality} shows that a block coordinate method is applicable to constrained problems, provided that the constraint set is separable over variable blocks. Secondly, we apply Remark 6 of the same paper with coefficient $\beta = 1/2$. At each step, we randomly choose one of two variable blocks, and factor $\beta$ rules the probability distribution. In Algorithm \ref{alg:ACRCD}, the probability of choosing block $\bz$ is $\eta = \sqrt{L_\bz}/(\sqrt{L_\bz} + \sqrt{L_\bs})$, and block $\bs$ is chosen with probability $(1 - \eta)$. Recall that for accuracy $\eps$ in primal problem \eqref{eq:problem_initial} we need accuracy $m\eps$ in dual problem \eqref{formula: original}. Combining the two remarks, we obtain that a resulting method makes $N$ iterations to reach $\eps$ accuracy with probability at least $1 - \delta$, where
\begin{align*}
	N = O\cbraces{\cbraces{\sqrt L_\bz + \sqrt L_\bs}\sqrt{\frac{R_\bz^2 + R_\bs^2}{m\eps}}\log\cbraces{\frac{1}{\delta}}}.
\end{align*}
Consequently, the expected number of computations of $\nabla H_\bz$ (that equals the number of communications) is $\eta N$, and the expected number of computations of $\nabla H_\bs$ (that corresponds to the number of local computations) is $(1 - \eta) N$. Substituting the expressions for $N$ and $\eta$, we obtain the desired result.
\end{proof}


\section{Conclusion}

In this paper, we considered a particular class of non-smooth decentralized problems. Due to specific problem structure we obtained methods that have a better dependency on problem complexity than general lower bounds. Our approach is based on passing to the dual problem. Moreover, we proposed two accelerated algorithms. The first algorithm is an accelerated primal-dual gradient method that is directly applied to the problem. The second method is a block-coordinate algorithm that allows to split communication and computation complexities.

\bibliographystyle{abbrv}
\bibliography{references}

\end{document}